\documentclass[envcountsame,envcountsect,runningheads]{llncs}
\usepackage{amssymb}
\usepackage{pmodels}
\usepackage[bookmarks=false]{hyperref}

\title{Presheaf models of quantum computation: \\an outline
}
\titlerunning{Presheaf models of quantum computation}  % abbreviated title (for running head)
\author{Octavio Malherbe\inst{1} \and Philip Scott\inst{2}
  \and Peter Selinger\inst{3}}
\authorrunning{Malherbe, Scott, and Selinger} % abbreviated author list (for running head)

\institute{IMERL-FING, Universidad de la Rep{\'u}blica, Montevideo, Uruguay\\
\email{malherbe@fing.edu.uy}\\
\and
Dept. of Mathematics and Statistics, University of Ottawa, Canada\\
 \email{phil@site.uottawa.ca}
\and Dept. of Mathematics and Statistics, Dalhousie University, Halifax, Canada
\email{selinger@mathstat.dal.ca}
}

\date{\today}

\input xy
\xyoption{all}

\begin{document}
\maketitle
\begin{abstract}
This paper outlines the construction of categorical models of
higher-order quantum computation. We construct a concrete denotational
semantics of Selinger and Valiron's quantum lambda calculus, which was
previously an open problem. We do this by considering presheaves
 over appropriate base categories arising from first-order
quantum computation. The main technical ingredients are Day's
convolution theory and Kelly and Freyd's notion of continuity of
functors. We first give an abstract description of the properties
required of the base categories for the model construction to work. We
then exhibit a specific example of base categories satisfying these
properties.

\end{abstract}

\section{Introduction}

Quantum computing is based on the laws of quantum physics. While no actual general-purpose quantum computer has yet been built, research in the last two decades indicates that quantum computers would be vastly more powerful than classical computers. For instance, Shor~\cite{Shor 94} proved in 1994 that the integer factoring problem can be solved in polynomial time on a quantum computer, while no efficient classical algorithm is known.

 Logic has played a key role in the development of classical computation theory, starting with the  foundations of  the subject in the 1930's by Church, G\"odel, Turing, and Kleene. For example, the pure untyped lambda calculus, one of the first models of computation invented by Church, can be simultaneously regarded as a prototypical functional programming language as well  as a formalism for denoting proofs. This is the so-called {\em proofs-as-programs} paradigm. Indeed, since the 1960's, various systems of typed and untyped lambda calculi have been developed, which on the one hand yield proofs in various systems of constructive and/or higher-order logic, while on the other hand denoting functional programs.  Modern programming languages such as  ML, Haskell, and Coq are often viewed in this light.

Recent research by Selinger, Valiron, and others \cite{SelVal 2006A,SelVal 2009} in developing ``quantum lambda calculi'' has shown that Girard's {\em linear logic} {\cite{GIRARD 87}} is a logical system that corresponds closely to the demands of quantum computation. Linear logic, a resource sensitive logic, turns out to formalize one of the central principles of quantum physics, the so-called {\em no-cloning property}, which asserts that a given unknown quantum state cannot be replicated. This property is reflected on the logical side by the requirement that a given logical assumption (or ``resource'') can only be used once. However, until now, the correspondence between linear logic and quantum computation has mainly been explored at the syntactic level.

In this paper we construct mathematical (semantic) models of higher-order quantum computation.
The basic idea is to start from existing low level models, such as the category of superoperators, and to use a Yoneda type presheaf
construction to adapt and extend these models to a higher order quantum setting.
To implement the latter, we use Day's theory of monoidal structure in presheaf categories, as well as the Freyd-Kelly  theory of continuous functors, to lift
 the required quantum structure {\cite{B.Day70,Freyd-Kelly 1972}}. Finally, to handle the probabilistic aspects of quantum  computation,  we employ Moggi's computational monads {\cite{Moggi88}}.

Our model construction depends on a sequence of
 categories and functors $\cB \rightarrow \cC\rightarrow \cD$, as well as a collection $\Gamma$ of cones
 in $\cD$. We use this data to obtain a pair of adjunctions
 \[\xymatrix{
[\cB^{op},\Set]\ar@<1ex>[rr]^{L}&&[\cC^{op},\Set]\ar@<1ex>[rr]^{F}\ar@<1ex>[ll]^{\Phi^{*}}_{\bot}&&[\cD^{op},\Set]_{\Gamma}
\ar@<1ex>[ll]^{G}_{\bot}}\]
   in which the left-hand adjunction gives an appropriate categorical model of the underlying linear logic, and the right-hand
adjunction gives a Moggi monad for probabilistic effects.
We then give sufficient conditions on $\cB \rightarrow \cC\rightarrow \cD$ and $\Gamma$ so that the
 resulting structure is a model of the quantum lambda calculus.
 One can describe various
 classes of concrete models by appropriate choices of diagrams $\cB \rightarrow \cC\rightarrow \cD$ and cones $\Gamma$.

 In this paper, we focus on the categorical aspects of
the model construction. Thus, we will not review the syntax of the
quantum lambda calculus itself (see~\cite{SelVal 2006A} and~\cite{SelVal 2009} for a quick review).
 Instead, we  take as our starting
point Selinger and Valiron's definition of a {\em categorical model of
the quantum lambda calculus}~\cite{SelVal 2009}. It was  proven
in~\cite{SelVal 2009} that the quantum lambda calculus forms an internal
language for the class of such models. This is similar to the well-known
interplay between typed lambda calculus and cartesian closed
categories~\cite{LambekScott86}. What was left open in~\cite{SelVal
2009} was the construction of a concrete such model (other than the one
given by the syntax itself). This is the question whose answer we sketch here.  Further
details can be found in the first author's PhD thesis \cite{Malherbe2010}.

\section{Categories of completely positive maps and superoperators}\label{COMP POSITIVE MAPS}\label{SUPEROPERATORS}

We first recall  various categories of finite dimensional Hilbert spaces
that we use in our study.
  Let $V$ be a finite dimensional Hilbert space, i.e., a finite
  dimensional complex inner product space. We write $\LL(V)$ for
  the space of linear functions $\rho:V\to V$.
 \begin{definition}\rm
   Let $V, W$ be finite dimensional Hilbert spaces. A linear function
   $F:\LL(V)\to\LL(W)$ is said to be {\em completely positive} if it
   can be written in the form
   $F(\rho) = \sum_{i=1}^{m} F_i\rho F_i^{\dagger}$,
   where $F_i:V\to W$ is a linear function and $F_i^{\dagger}$ denotes the linear adjoint of $F_i$ for $i=1,\ldots,m$.
 \end{definition}

 \begin{definition}\rm
   The category $\CPM_s$ of {\em simple completely positive maps} has
   finite dimensional Hilbert spaces as objects, and the morphisms
   $F:V\to W$ are completely positive maps $F:\LL(V)\to\LL(W)$.
 \end{definition}

 \begin{definition}\label{COMPLETE POSITIVE MAPS}\rm
   The category $\CPM$ of {\em completely positive maps} is defined as
   $\CPM=\CPM_s^{\oplus}$, the biproduct completion of $\CPM_s$.
   Specifically, the objects of $\CPM$ are finite sequences
   $(V_1,\ldots,V_n)$ of finite-dimensional Hilbert spaces, and a
   morphism $F:(V_1,\ldots,V_n)\to (W_1,\ldots,W_m)$ is a matrix
   $(F_{ij})$, where each $F_{ij}:V_j\to W_i$ is a completely positive
   map. Composition is defined by matrix multiplication.
 \end{definition}

 \begin{remark}
The category $\CPM$ is the same (up to equivalence) as
   the category $\W$ of~\cite{Sel 2004} and the category
   $\CPM(\FdHilb)^{\oplus}$ of~\cite{Sel 2005}.
 \end{remark}

 Note that for any two finite dimensional Hilbert spaces $V$ and $W$,
 there is a canonical isomorphism
 $\phi_{V,W} : \LL(V\otimes W) \to \LL(V)\otimes \LL(W).$

 \begin{remark}
 The categories $\CPM_s$ and $\CPM$ are symmetric monoidal.
 For $\CPM_s$, the tensor product is given on objects by the tensor
 product of Hilbert spaces $V\bar{\otimes} W = V\otimes W$, and
 on morphisms by the following induced map $f\bar{\otimes} g$:=
 $\LL(V\otimes W) \stt{\phi_{V,W}}
 \LL(V)\otimes \LL(W)\stt{f\otimes g}  \LL(X)\otimes \LL(Y)
 \stt{\phi^{-1}_{X,Y}}\LL(X \otimes Y)$.
 The remaining structure (units, associativity, symmetry  maps) is
 inherited from Hilbert spaces.   Similarly, for the
symmetric monoidal structure on $\CPM$, define
$(V_i)_{i\in I} \otimes (W_j)_{j\in J} = (V_i\otimes W_j)_{i\in I, j\in J}.$
 This extends to morphisms in an obvious way. For details, see~\cite{Sel 2004}.
\end{remark}

\begin{definition}\rm
 We say that a linear map $F:\LL(V)\rightarrow \LL(W)$
is  \textit{trace preserving}  when it satisfies
$\tr_W(F(\rho))=\tr_V(\rho)\label{TRACE PRESERVING CONDITION}$
for all positive $\rho\in \LL(V).$
$F$ is called {\em trace
non-increasing} when it satisfies
$\tr_W(F(\rho))\leq \tr_V(\rho)\label{TRACE NON-INCREASING CONDITION}$
for all positive $\rho \in \LL(V).$
\end{definition}

\begin{definition}\rm
 A linear map $F :\LL(V)\rightarrow \LL(W)$ is called a
 {\em trace preserving superoperator} if it is completely positive and
 trace preserving, and it is  a {\em trace non-increasing
 superoperator} if it is completely positive and trace non-increasing.
 \end{definition}

\begin{definition}\rm
  A completely positive map $F:(V_1,\ldots,V_n) \to
 (W_1,\ldots,W_m)$ in the category $\CPM$ is called a {\em trace
 preserving superoperator} if for all $j$ and all positive $\rho\in \LL(V_j)$,
      $ \sum_i \tr(F_{ij}(\rho)) = \tr(\rho),$
 and a {\em trace non-increasing superoperator} if for all $j$ and all
 positive $\rho\in \LL(V_j)$,
       $\sum_i \tr(F_{ij}(\rho)) \leq \tr(\rho).$
 \end{definition}
 We now define four symmetric monoidal categories of
 superoperators. All of them are symmetric monoidal subcategories of $\CPM$.

\noindent\begin{minipage}{\textwidth}
 \begin{definition}\label{Q OPLUS CATEGORY}\label{Q SIMPLE SUPEROPERATORS}\label{TRACE PRESERVING CATEGORY}\ \rm
  \begin{itemize}
 \item $\QQ_s$ and $\QQ'_s$ have the same objects as $\CPM_s$, and $\QQ$ and $\QQ'$ have the
 same objects as $\CPM$.
 \item The morphisms of $\QQ_s$  and $\QQ$ are trace non-increasing superoperators,
   and the morphisms of $\QQ'_s$ and $\QQ'$ are trace preserving
   superoperators.
 \end{itemize}
\end{definition}
\end{minipage}

\begin{remark} The categories $\QQ_s$, $\QQ$, $\QQ'_s$, and $\QQ'$ are all symmetric
 monoidal. The symmetric monoidal structure is inherited from $\CPM_s$ and $\CPM$, respectively, and
 it is easy to check that all the structural maps are trace
 preserving.
 \end{remark}

\begin{lemma}
$\QQ$ and $\QQ'$ have finite coproducts.
\end{lemma}
\begin{proof}

The injection/copairing maps are as in $\CPM$ and are
 trace preserving.
\end{proof}

\section[Presheaf models]
         {Presheaf models of a quantum lambda calculus}\label{PREASHEAVES MODELS OF Q C}

Selinger defined an elementary quantum flow chart language in \cite{Sel 2004}, and gave a denotational model in terms of superoperators. This axiomatic framework captures the behaviour and interconnection between the basic concepts of quantum computation  (for example, the manipulation of quantum bits under the basic operations of measurement and unitary transformation) in a lower-level language. In particular, the semantics of this framework is very well understood: each program corresponds to a concrete superoperator.

Higher-order functions are functions that can input or output other functions. In order to deal with such functions, Selinger and Valiron introduced, in a series of papers \cite{SelVal 2006B,SelVal 2008,SelVal 2009}, a typed lambda calculus for quantum computation and investigated several aspects of its semantics. In this context, they combined two well-established areas:  the intuitionistic fragment of Girard's linear logic~\cite{GIRARD 87} and Moggi's computational monads ~\cite{Moggi88}.

The type system of Selinger and Valiron's quantum lambda calculus is based on intuitionistic linear logic. As is usual in linear logic, the logical rules of weakening and contraction are introduced in a controlled way by an operator ``$!$'' called ``of course'' or ``exponential''. This operator creates a bridge between two different kinds of computation. More precisely, a value of a general type $A$ can only be used once,
 whereas a value of type $!A$ can be copied and used multiple
 times. As mentioned in the introduction, the impossibility of copying quantum information is one of the
 fundamental differences between quantum information and classical
 information, and is known as the {\em no-cloning property}.
 From a logical perspective, this is related to the failure of the contraction rule; thus it seems natural to use linear logic in discussing quantum computation.  It is also well known that  categorically, the operator ``$!$'' satisfies the properties of a {\em comonad} (see~\cite{MELLIES}).

Since the quantum lambda calculus has higher-order functions, as well as probabilistic operations
(namely measurements), it must be equipped with an evaluation strategy in order to be consistent. Selinger and Valiron addressed this by choosing the {\em call-by-value} evaluation strategy. This introduces a distinction between {\em values} and {\em computations}. At the semantic level, Moggi \cite{Moggi88} proposed using the notion of monad as an appropriate tool for interpreting computational behaviour. In our case, this will be a strong monad.

So let us now describe Selinger and Valiron's notion  of a {\em categorical model
 of the quantum lambda calculus} \cite{SelVal 2009}.

\subsection{Categorical models of the quantum lambda calculus}\label{CAT MODEL QLC}

  In what follows,
let $(\cC,\otimes,I,\alpha,\rho,\lambda,\sigma)$ be a symmetric monoidal category {\cite{MacLane 91}}.
\begin{definition}
\rm
A \textit{symmetric monoidal comonad} $(!,\delta,\varepsilon,m_{A,B},m_I)$ is a comonad $(!,\delta,\varepsilon)$ where the functor $!$ is a monoidal functor $(!,m_{A,B},m_I)$, i.e., with natural transformations $m_{A,B}:{!}A\otimes{!}B\to{!}(A\otimes B)$ and $m_I:I\to{!}I$, satisfying appropriate coherence axioms~\cite{KOCK70}
 such that  $\delta$ and $\varepsilon$ are symmetric monoidal natural transformations.
\end{definition}

\begin{definition}
\label{LINEAR EXPONENTIAL COMONAD}
\rm
A \textit{linear exponential comonad} is a symmetric monoidal comonad $(!,\delta,\varepsilon,m_{A,B},m_I)$ such that for every $A\in \cC$, there exists a commutative comonoid $(A,d_A,e_A)$ satisfying some technical requirements (see~\cite{BIERMAN95,SelVal 2009}).
\end{definition}

\begin{definition}
\rm
Let $(T,\eta,\mu)$ be a strong monad on $\cC$. We say that $\cC$ has \textit{Kleisli exponentials} if there exists a functor $[-,-]_{k}:\cC^{op}\times \cC\rightarrow \cC$ and a natural isomorphism:
$\cC(A\otimes B,TC)\cong\cC(A,[B,C]_{k}).$
\end{definition}

\begin{definition}[Selinger and Valiron~\cite{SelVal 2009}]
\label{LINEAR CATEGORY FOR DUPLICATION}
\rm
A \textit{linear category for duplication} consists of a symmetric monoidal category $(\cC,\otimes,I)$ with the following structure:
\begin{itemize}
\item an idempotent, strongly monoidal, linear exponential comonad $(!,\delta,\varepsilon,d,e)$,
\item a strong monad $(T,\mu,\eta,t)$ such that $\cC$ has Kleisli exponentials.
\end{itemize}
\end{definition}
Further, if the unit $I$ is a terminal object we shall speak of an {\em affine linear category for duplication}.
\begin{remark}
\label{ADJUNCTION PRESENTATION}
\rm
 Perhaps surprisingly, following the work of Benton, a linear category for duplication can be obtained from a structure that is much easier to describe, namely, a pair of monoidal adjunctions {\cite{BENTON 1994,MELLIES,KOCK72}}
\vspace{-1ex}
\[\xymatrix{
(\mathcal{B},\times,1)\ar@<1ex>[rr]^{(L,l)}&&(\mathcal{C},\otimes,I)\ar@<1ex>[rr]^{(F,m)}\ar@<1ex>[ll]^{(I,i)}_{\bot}&&(\mathcal{D},\otimes,I),\ar@<1ex>[ll]^{(G,n)}_{\bot}}\]
where the category $\cB$ has finite products and $\cC$ and $\cD$ are symmetric monoidal closed.
The monoidal adjoint pair of functors on the left represents a linear-non-linear model of linear logic in the sense of Benton~\cite{BENTON 1994}, in which we obtain a monoidal comonad by setting $!=L\circ I$. The monoidal adjoint pair on the right gives rise to  a strong monad $T=G\circ F$ in the sense of Kock~\cite{KOCK70,KOCK72}, which is also a computational monad in the sense of Moggi~\cite{Moggi88}.
\end{remark}
We now state the main definition of a model of the quantum lambda calculus.

\begin{definition}[Models of the quantum lambda calculus~\cite{SelVal 2009}]
\label{DEF MODEL OF QUANTUM LAMBDA CALCULUS}
\rm
An \textit{abstract model of the quantum lambda calculus} is an affine linear category for duplication $\cC$ with finite coproducts, preserved by the comonad $!$. A {\em concrete model of the quantum lambda calculus} is an abstract model  such  that there exists a full and faithful embedding $\QQ\hookrightarrow \cC_{T}$, preserving tensor $\otimes$ and coproduct $\oplus$ up to isomorphism, from the category $\QQ$ of norm non-increasing superoperators (see Definition~\ref{Q OPLUS CATEGORY}) into the Kleisli category of the monad $T$.
\end{definition}
\begin{remark} To make the connection to quantum lambda calculus: the
 category $\cC$, the Kleisli category $\cC_T$, and the co-Kleisli
 category $\cC_{!}$ all have the same objects, which correspond to
 {\em types} of the quantum lambda calculus. The morphisms $f:A\to B$
 of $\cC$ correspond to {\em values} of type $B$ (parameterized by
 variables of type $A$). A morphism $f:A\to B$ in $\cC_T$, which is
 really a morphism $f:A\to TB$ in $\cC$, corresponds to a {\em
 computation} of type $B$ (roughly, a probability distribution of
 values). Finally, a morphism $f:A\to B$ in $\cC_{!}$, which is really
 a morphism $f:{!A}\to B$ in $\cC$, corresponds to a {\em classical
 value} of type $B$, i.e., one that only depends on classical
 variables. The idempotence of ``$!$'' implies that morphisms ${!A}\to
 B$ are in one-to-one correspondence with morphisms ${!A}\to{!B}$,
 i.e., classical values are duplicable. For details, see
 \cite{SelVal 2009}.
 \end{remark}

\subsection{Outline of the procedure for obtaining a concrete model}\label{OUTLINE OF THE PROCEDURE}

 We construct the model in two stages. The first (more elaborate) stage
 constructs {\em abstract} models by
 applying certain general presheaf constructions to diagrams of
 functors $\cB\rightarrow\cC\rightarrow\cD.$
 In Section~\ref{MAIN THEOREM OF CHAPTER 4}
 we find the precise conditions required on diagrams $\cB\rightarrow\cC\rightarrow\cD$
to obtain a valid abstract
 model.
In the second stage, we  construct a {\em concrete} model of the
 quantum lambda calculus by identifying particular base categories so that  the remaining conditions of Definition~\ref{DEF MODEL OF QUANTUM LAMBDA CALCULUS} are satisfied. This
 is the content of Sections~\ref{Q''CATEGORY AND FUNCTORS PSI AND PHI} and~\ref{A CONCRETE MODEL}.

 We divide the two stages of construction into eight main steps.

\begin{enumerate}
\item The basic idea of the construction is to lift a sequence of functors
 $\cB\stackrel{\Phi}\rightarrow\cC\stackrel{\Psi}\rightarrow\cD$  into a pair of adjunctions between presheaf categories
      \[\xymatrix{
[\cB^{op},\Set]\ar@<1ex>[rr]^{L}&&[\cC^{op},\Set]\ar@<1ex>[rr]^{F_1}\ar@<1ex>[ll]^{\Phi^{*}}_{\bot}&&[\cD^{op},\Set].
\ar@<1ex>[ll]^{\Psi^{*}}_{\bot}}\]

   Here, $\Phi^{*}$ and $\Psi^{*}$ are the precomposition functors, and $L$ and
  $ F_1$ are their left Kan extensions.  By Remark~\ref{ADJUNCTION PRESENTATION}, such a pair of
   adjunctions potentially yields a linear category for duplication,
   and thus, with additional conditions, an abstract model of
   quantum computation. Our goal is to identify the particular
   conditions on $\cB$, $\cC$, $\cD$, $\Phi$, and $\Psi$ that make this construction
   work.
  \item By Day's convolution construction (see~\cite{B.Day70}), the requirement that $[\cC^{op},\Set]$ and $[\cD^{op},\Set]$
   are monoidal closed can be achieved by requiring $\cC$ and $\cD$ to be
   monoidal. The requirement that the adjunctions $L \dashv \Phi^{*}$ and
   $F_1 \dashv \Psi^{*}$ are monoidal is directly related to the fact
   that the functors $\Psi$ and $\Phi$ are strong monoidal. More precisely,
   this implies that the left Kan extension is a strong monoidal
   functor~\cite{B.Day-Street95} which in turn determines the enrichment of the adjunction~\cite{KELLY74}.
   We also note that the category $\cB$ must be cartesian.

    \item One important complication with the model, as discussed so far, is
   the following. The Yoneda embedding $Y : \cD\rightarrow [\cD^{op},\Set]$ is full
   and faithful, and by Day's result, also preserves the monoidal
   structure $\otimes$. Therefore, if one takes $\cD=\QQ$, all but one of
   the conditions of a concrete model (from Definition~\ref{DEF MODEL OF QUANTUM LAMBDA CALCULUS}) are
   automatically satisfied. Unfortunately, however, the Yoneda embedding does
   not preserve coproducts, and therefore the remaining condition of
   Definition~\ref{DEF MODEL OF QUANTUM LAMBDA CALCULUS} fails. For this reason, we modify the construction
   and use a modified presheaf category with a coproduct preserving
   Yoneda embedding.
   More specifically, we choose a set $\Gamma$ of cones in $\cD$, and use the theory of continuous functors by
   Lambek~\cite{Lambek66} and Freyd and Kelly~\cite{Freyd-Kelly 1972}
   to construct a reflective subcategory $[\QQ^{op},\Set]_{\Gamma}$ of
   $[\QQ^{op},\Set]$, such that the modified Yoneda embedding $\QQ \to
   [\QQ^{op},\Set]_{\Gamma}$ is coproduct preserving.
   Our adjunctions, and the
   associated Yoneda embeddings, now look like this:

      $$\xymatrix@=25pt{
[\cB^{op},\Set]\ar[rr]^{L \,\dashv\, \Phi^{*}}&& [\cC^{op},\Set]\ar[rr]^{F \,\dashv\, G}&&[\cD^{op},\Set]_{\Gamma}\\
\cB\ar[u]^{Y}\ar[rr]^{\Phi}&&\cC\ar[u]^{Y}\ar[rr]^{\Psi}&&\cD\ar[u]^{Y_{\Gamma}}
}$$

   The second pair of adjoint functors $F \dashv G$ is itself generated by the
   composition of two adjunctions:
   \[\xymatrix{
[\cC^{op},\Set]\ar@<1ex>[rr]^{F_1}&&[\cD^{op},\Set]\ar@<1ex>[rr]^{F_2}\ar@<1ex>[ll]^{\Psi^{*}}_{\bot}&&[\cD^{op},\Set]_{\Gamma}\ar@<1ex>[ll]^{G_2}_{\bot}}\]
   Here $\cD = \QQ$ and the pair of functors $F_2 \dashv G_2$ arises from the reflection of
   $[\QQ^{op},\Set]_{\Gamma}$ in $[\QQ^{op},\Set]$.
   The structure of the modified
   Yoneda embedding $\QQ \to [\QQ^{op},\Set]_{\Gamma}$ depends crucially on general
   properties of functor categories {\cite{Lambek66,Freyd-Kelly 1972}}. Full details are given in~\cite{Malherbe2010}.

   \quad To ensure that
   the reflection functor remains strongly monoidal, we will use
   Day's reflection theorem~\cite{B.Day72},
   which yields necessary conditions
   for the reflection to be strong monoidal, by inducing a monoidal
   structure from the category $[\QQ^{op},\Set]$ into its subcategory $[\QQ^{op},\Set]_{\Gamma}$.
   In particular, this induces a
   constraint on the choice of $\Gamma$: all the cones
   considered in $\Gamma$ must be preserved by the opposite functor of
   the tensor functor in $\cD$.

   \item Notice that the above adjunctions are examples of what in topos
   theory are called  {\em essential geometric} morphisms, in which
   both functors are left adjoint to some other two functors: $L \dashv
   \Phi^{*} \dashv \Phi_{*}$. Therefore, this shows that the comonad ``$!$''
   obtained will preserve finite coproducts.

   \item The condition for the comonad ``$!$'' to be idempotent turns out to
   depend on the fact that the functor $\Phi$ is full and faithful; see Section~\ref{idemp comonad}.

 \item In addition to requiring that ``$!$'' preserves coproducts, we
 also need ``$!$'' to preserve the tensor, i.e., to be strongly
 monoidal, as required in Definition~\ref{DEF MODEL OF QUANTUM LAMBDA CALCULUS}.
This property is unusual for models of intuitionistic linear logic and restricts the possible choices  for the category $\cC$. In brief, since the left Kan extension along $\Phi$ is a strong monoidal functor, we find a concrete condition on the category $\cC$ that is necessary to ensure that the property holds when we lift the functor $\Phi$ to the category of presheaves; see Section~\ref{A STRONGLY COMONAD}.

  \item
  Our next task is to translate these categorical properties to the Kleisli category.
We use the comparison Kleisli functor to pass from the framework we have already established to the Kleisli monoidal adjoint pair of functors. Also, at the same time, we shall find it convenient to characterize the  functor $H:\cD\rightarrow [\cC^{op},\Set]_T$ as a strong monoidal functor.  The above steps yield an abstract model of quantum
 computation, parameterized by $\cB\rightarrow \cC \rightarrow \cD$ and~$\Gamma$.

 \item Finally,  in Section~\ref{Q''CATEGORY AND FUNCTORS PSI AND PHI}, we will identify specific
   categories $\cB$, $\cC$, and $\cD$ that yield a concrete model of quantum
   computation. We let $\cD=\QQ$, the category of superoperators.
   We let $\cB$ be the category of
   finite sets.  Alas, identifying a suitable candidate for   $\cC$ is difficult.
For example,  two   requirements are that  $\cC$ must be affine monoidal and  must satisfy the condition of equation~(\ref{EQUATION MULTIPLICATIVE KERNEL}) in Section~\ref{A STRONGLY COMONAD} below.
   We construct such a  $\cC=\QQ''$
   related to the category $\QQ$ of superoperators with the help of some universal
   constructions.

\end{enumerate}
The above base category $\QQ''$ plays a central role in our construction.
While the higher-order structural properties of the quantum lambda calculus hold at the pure   functor category level,   the interpretation of concrete quantum operations takes place mostly at this base level.

Let us now discuss some details of the construction.

\subsection{Categorical models of linear logic on presheaf categories}\label{CAT MODELS OF LL}

The first  categorical models of linear logic were given by Seely~\cite{SEELY}.
The survey by Melli\`{e}s is an excellent introduction~\cite{MELLIES}. Current state-of-the-art
 definitions  are Bierman's definition of a linear category~\cite{BIERMAN95},
 simplified yet more  by Benton's definition of a linear-non-linear category (\cite{BENTON 1994}, cf. Remark~\ref{ADJUNCTION PRESENTATION} above). Benton proved the equivalence of these two notions {\cite{BENTON 1994,MELLIES}}.

\begin{definition}[Benton~\cite{BENTON 1994}]
\label{DEF L-N-L BENTON MODEL}\rm
A \textit{linear-non-linear category} consists of:
\begin{enumerate}
\item[(1)] a symmetric monoidal closed category $(\mathcal{C},\otimes ,I, \multimap )$,
\item[(2)] a category $(\mathcal{B},\times,1)$ with finite products,
\item[(3)] a symmetric monoidal adjunction:
$\xymatrix@1{
(\mathcal{B},\times,1)\ar@<1ex>[rr]^{(F,m)}&&(\mathcal{C},\otimes,I)\ar@<1ex>[ll]^{(G,n)}_{\bot}}.$
\end{enumerate}
\end{definition}

\begin{remark}
We use Kelly's characterization of monoidal adjunctions
to simplify condition $(3)$ in Definition~\ref{DEF L-N-L BENTON MODEL} above to:
\begin{enumerate}
\item[(3')] an adjunction:
$\xymatrix@1@C+=5em{
(\mathcal{B},\times,1)\ar@<1ex>[r]^{F}&(\mathcal{C},\otimes,I),\ar@<1ex>[l]^{G}_{\bot}}$
and there exist isomorphisms
$m_{A,B}:FA\otimes FB\rightarrow F(A\times B)$ and $m_I:I\rightarrow F(1)$,
making $(F,m_{A,B},m_I):(\mathcal{B},\times,1)\rightarrow (\mathcal{C},\otimes ,I) $ a strong symmetric monoidal functor.
\end{enumerate}
Details of this characterization can found in~\cite{MELLIES}.
\end{remark}

We can characterize Benton's linear-non-linear models (Definition~\ref{DEF L-N-L BENTON MODEL}) on presheaf categories using Day's monoidal structure~\cite{B.Day70}.
This is an application of monoidal enrichment of the Kan
 extension,  see~\cite{B.Day-Street95}. We use the following:

\begin{proposition}[Day-Street\cite{B.Day-Street95}]
Suppose we have a strong monoidal functor $\Phi:(\mathcal{A},\otimes,1)\rightarrow(\mathcal{B},\otimes,I)$ between two monoidal categories,
 i.e., we have natural isomorphisms
 $\Phi(a)\otimes\Phi(b)\cong\Phi(a\otimes b)$ and $I\cong \Phi(I)$. Consider the left Kan extension along $\Phi$ in the functor
category $[\cB^{op},\Set]$, where the copower is the cartesian product on sets:
$Lan_{\Phi}(F)=\int^{a}\mathcal{B}(-,\Phi(a))\times F(a).$
Then $Lan_{\Phi}$ is strong monoidal.
\end{proposition}

\begin{remark}\rm
If $\cA$ is cartesian then the Day tensor (convolution) $ [\mathcal{A}^{op},\Set]\times[\mathcal{A}^{op},\Set]\stt{\otimes_D}[\mathcal{A}^{op},\Set]$ is a pointwise product of functors. Also
if the unit of a monoidal category $\cC$ is a terminal object then the unit of $\otimes_D$ is also terminal.
\end{remark}

\subsection{Idempotent comonad in the functor category}\label{idemp comonad}
A comonad $({!},\epsilon,\delta)$ is said to be {\em idempotent} if $\delta:{!}\Rightarrow {!!}$ is an isomorphism.
Let $(!,\epsilon,\delta)$ be the comonad generated by an adjunction
$\xymatrix@1{(\mathfrak{B},\times ,1)\ar@<1ex>[r]^<>(.5){F}& (\mathfrak{C},\otimes,I). \ar@<1ex>[l]^<>(.5){G}_<>(.5){\bot}}$
Then $\delta=F\eta_G$ with $\eta:I\rightarrow GF$.
Thus if $\eta$ is an isomorphism then $\delta$ is also an isomorphism. In the context of our model construction, how can we
guarantee that $\eta$ is an isomorphism? Consider the unit
$\eta_{B}:B\Rightarrow \Phi^{*}(Lan_{\Phi}(B))$ of the adjunction generated by the Kan extension:
$$\xymatrix@1{[\cB^{op},\Set]\ar@<1ex>[r]^<>(.5){Lan_{\Phi}}& [\cC^{op},\Set]. \ar@<1ex>[l]^<>(.5){\Phi^{*}}_<>(.5){\bot}}$$

\begin{proposition}[\cite{Borceux94}]\label{F FULL AND FAITHFUL ETA ISO}
If $\Phi$ is full and faithful then $\eta_B:B\Rightarrow \Phi^{*}(Lan_{\Phi}(B))$ is an isomorphism.
\end{proposition}

\subsection{A strong comonad}
\label{A STRONGLY COMONAD}
In this section we study conditions that force the idempotent comonad above to be a strong monoidal functor. This property is part of the model we are building and is one of the main differences with previous models of  intuitionistic linear logic \cite{MELLIES}.

To achieve this, consider a fully faithful functor
$\Phi:\cB\to\cC$, as in Section~\ref{OUTLINE OF THE PROCEDURE}.
Let $[\cC^{op},\Set]\stackrel{\Phi^{*}}\longrightarrow[\cB^{op},\Set] $
be the precomposition functor, i.e., the right adjoint of the left Kan extension.
 \begin{lemma}[\cite{B.Day06}]\label{MULTIPLICATIVE KERNEL CONDITIONS}
If there exists a natural isomorphism
\begin{equation}
\label{EQUATION MULTIPLICATIVE KERNEL}
\cC(\Phi(b),c)\times\cC(\Phi(b),c')\cong\cC(\Phi(b),c\otimes c'),
\end{equation}
where $b\in \cB$ and $c,c'\in\cC$ and $\Phi$ is a fully faithful
functor satisfying $\Phi(1)=I$,
then $\Phi^{*}$ is a strong monoidal functor.
\end{lemma}

In Section~\ref{Q double prime AND CATEGORY AND FUNCTOR PSI} we shall build a category satisfying this specific requirement among others. More precisely, from our viewpoint, this will depend on the construction of a certain category that we will name $\QQ''$, which is a modification of the category $\QQ$ of superoperators. Also, we will consider a fully faithful strong monoidal functor  $\Phi:(\FinSet,\times,1)\rightarrow (\cC,\otimes_{\cC},I)$ that generates the first adjunction in Section~\ref{OUTLINE OF THE PROCEDURE}, where $\cC=\QQ''$.

\subsection{The functor $H:\cD\rightarrow \hat{\cC}_{T}$}\label{COPRODUCT INDUCED KLEISLI}
Let $\cC$ and $\cD$ be categories.
Consider an adjoint pair of functors
$\xymatrix@1{
[\cC^{op},\Set]\ar@<1ex>[r]^<>(.5){F}&[\cD^{op},\Set]_{\Gamma},\ar@<1ex>[l]^<>(.5){G}_<>(.5){\bot}}$ as mentioned in Section \ref{OUTLINE OF THE PROCEDURE}, item $3$.
Let  $T=G\circ F$ and $\hat\cC=[\cC^{op},\Set]$.
 In this section we consider the construction of a coproduct preserving and tensor preserving functor $H:\cD\rightarrow \hat{\cC}_{T}$ with properties similar to the Yoneda embedding, for a
 general category $\cD$.

Let  $F_{1}\dashv G_{1}$ and $F_{2}\dashv G_{2}$ be two monoidal adjoint pairs with associated natural transformations $(F_{1},m_{1})$, $(G_{1},n_{1})$ and $(F_{2},m_{2})$, $(G_{2},n_{2})$.
We shall use the following notation: $F=F_2\circ F_1$, $G=G_1\circ G_2$, and $T=G\circ F$.
We now describe a typical situation of this kind generated by a functor $\Psi:\cC\rightarrow\cD$.

Let us consider $F_1=Lan_{\Psi}$ and $G_1=\Psi^{*}$. With some co-completeness condition assumed, we can express
$F_1(A)=\int^{c}\cD(-,\Psi(c))\otimes A(c)$ and $G_1=\Psi^{*}$.

On the other hand we consider
$F_2=Lan_{Y}(Y_{\Gamma}):[\cD^{op},\Set]\rightarrow[\cD^{op},\Set]_{\Gamma}$, where
$Y_{\Gamma}:\cD\rightarrow [\cD^{op},\Set]_{\Gamma}$ is the co-restriction of the Yoneda functor given by $Y_{\Gamma}(d)=\cD(-,d)$. Thus we have
$F_2(D)=\int^{d}D(d)\otimes Y_{\Gamma}(d)$. Assuming that $[\cD^{op},\Set]_{\Gamma}$ is co-complete and contains the representable presheaves, then the right adjoint $G_2$ is isomorphic to the inclusion functor.

\vspace{-1ex}
 \subsubsection{Definition of $H$.}\label{CANONICAL CHOICE} \

\vspace{1ex}

\noindent
We want to study the following situation:

$$
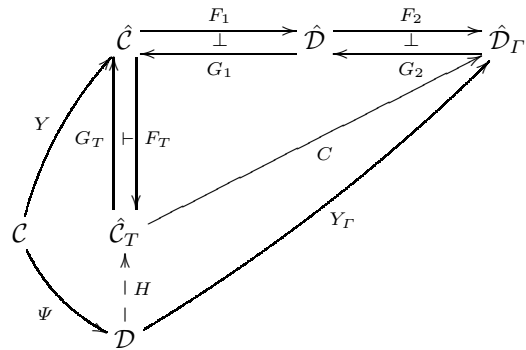
\xymatrix@=25pt{
&\hat{\cC}\ar@<1ex>[dd]^{F_{T}}\ar@<1ex>[rr]^{F_{1}}&&\hat{\cD}\ar@<1ex>[ll]^{G_{1}}_{\bot}\ar@<1ex>[rr]^{F_{2}}&&\hat{\cD}_{\Gamma}\ar@<1ex>[ll]^{G_{2}}_{\bot} \\
 &&&&&\\
 \cC\ar@/_/[rd]_{\Psi}\ar@/^/[ruu]^{Y}&\hat{\cC}_{T}\ar@<1ex>[uu]^{G_{T}}_{\vdash}\ar[rrrruu]_{C}&&&&&\\
 &\cD\ar@{-->}[u]_{H}\ar@/_/[rrrruuu]_{Y_{\Gamma}}&&&&&
}$$

The goal is to determine a full and faithful functor, denoted $H$ in this diagram, that preserves tensor and coproduct.

First, notice that the perimeter of this diagram
commutes on objects:
$F_1(\cC(-,c))=$
$\int^{c'}\cD(-,\Psi(c'))\otimes \cC(c',c)=\cD(-,\Psi(c))$
and when we evaluate again, using $F_2$, we obtain:
$$F_2(\cD(-,\Psi(c)))=\int^{d'}\cD(d',\Psi(c))\otimes Y_{\Gamma}(d')=Y_{\Gamma}(\Psi(c))=\cD(-,\Psi(c)).$$
Summing up, we have that $F(\cC(-,c))=\cD(-,\Psi(c))$ up to isomorphism.

Suppose now that $\Psi$ is essentially onto on objects and we have that:
$$\cD(-,d)\cong\cD(-,\Psi(c))$$
for some $c\in\cC$, i.e.,
 we can make a choice, for every
 $d\in |\cD|$, of some $c\in |\cC|$ such that $\Psi(c)\cong d$. Let us
 call this choice a ``choice of preimages''. We can therefore define
 a map $H:|\cD|\rightarrow |\hat{\cC}_{T}|$ by $H(d)=\cC(-,c)$ on objects.

Then we can define a functor $H:\cD\rightarrow \hat{\cC}_{T}$ in the following way:
 let $d\stackrel {f}\rightarrow d'$ be an arrow in the category $\cD$. We apply $Y_{\Gamma}$ obtaining $\cD(-,d)\stt{Y_{\Gamma}(f)}\cD(-,d')$. This arrow is equal to $\cD(-,\Psi(c))\stt{Y_{\Gamma}(f)}\cD(-,\Psi(c'))$ for some $c,c'\in\cC$, and for the reason stipulated above is equal to  $F(\cC(-,c))\stackrel {Y_{\Gamma}(f)}\longrightarrow F(\cC(-,c'))$. Now we use the fact that the comparison functor
$C:\hat{\cC}_{T}\rightarrow \hat{\cD}_{\Gamma}$, i.e.,
$$C:\hat{\cC}_{T}(\cC(-,c),\cC(-,c'))\rightarrow \hat{\cD}_{\Gamma}(F(\cC(-,c)),F(\cC(-,c'))),$$
is full and faithful.  Thus there is a unique $\gamma:\cC(-,c)\rightarrow\cC(-,c')$ such that $C(\gamma)=Y_{\Gamma}(f)$. Then we can define $H(f) =\gamma$ on morphisms and (as mentioned above) $H(d)=\cC(-,c)$ on objects, where $c$ is given by our choice of preimages.\\

\noindent
{\bf $C:\hat{\cC}_{T}\rightarrow\hat{\cD}_{\Gamma}$ is a strong monoidal functor}\\

\noindent
We define $C(A)\otimes_{\hat{\cD}_{\Gamma}}C(B)\stackrel {u_{AB}}\longrightarrow C(A\otimes_{\cC_{T}}B)$ as
$F(A)\otimes_{\hat{\cD}_{\Gamma}}F(B)\stackrel {m_{AB}}\longrightarrow F(A\otimes B)$. It is easy to check naturality.
Also define $I\stackrel {u_{I}=m_{I}}\longrightarrow C(I)=F(I)$.
Since $m_{AB}$ and $m_{I}$ are invertible in $\hat{\cD}_{\Gamma}$, we have that $u_{AB}$ and $u_{I}$ are invertible. This implies that $(C,m)$ is a strong functor.
Also, coherence of isomorphisms is easily checked.

 \vspace{1ex}

\noindent
{\bf $H: {\cD}\rightarrow\hat{\cC}_{T}$ is a strong monoidal functor}

\noindent
We want to define a natural transformation $H(A)\otimes_{\hat{\cC}_{T}} H(B)\stackrel {\psi_{A,B}}\longrightarrow H(A\otimes_{\cD} B)$ that makes $H$ into a strong monoidal functor.

We begin by recalling that $(C,u)$ and $(Y_{\Gamma},y)$ are strong monoidal, i.e., $u$ and $y$ are isomorphisms. Since $C$ is fully faithful, this allows us to define $\psi_{A,B}$ as the unique map making the following diagram commute:
$$\xymatrix@=25pt{
 Y_{\Gamma}(A)\otimes Y_{\Gamma}(B)\ar[rr]^{y_{A,B}}\ar[dr]^{u_{HA,HB}}&&Y_{\Gamma}(A\otimes B)=C\circ H(A\otimes B).\\
 &C(H(A)\otimes H(B))\ar[ru]^{C(\psi_{A,B})}&
}$$
We define $\psi_{I}$ similarly.
Furthermore, $\psi$ is a natural transformation and satisfies all the axioms of a monoidal structure. We refer to~\cite{Malherbe2010} for the details.
\vspace{1ex}

\noindent
 {\bf $H$ preserves coproducts}

 \vspace{1ex}

\noindent
 Here we focus on the specific problem of the preservation of finite coproducts by the functor $H$ defined in Section~\ref{CANONICAL CHOICE}.
First,  note that the category $[\cC^{op},\Set]$ has finite coproducts, computed pointwise.
Moreover, the Kleisli category $\hat{\cC}_{T}$ inherits the coproduct structure from $\hat{\cC}$ since:
 \begin{proposition}
 If $\cC$ has finite coproducts, then so does $\cC_{T}$.
\end{proposition}

 Therefore, $[\cC^{op},\Set]_{T}$ has finite coproducts.
Recall that the comparison functor
$C:[\cC^{op},\Set]_{T}\rightarrow[\cD^{op},\Set]_{\Gamma}$ is fully faithful.
Also, by a well known property of representable functors (see~\cite{Lambek66}), we have that
$H:\cD\rightarrow [\cC^{op},\Set]_{T}$ preserves coproducts iff $[\cC^{op},\Set]_{T}(H-,A):\cD^{op}\rightarrow\Set$ preserves products for every $A\in [\cC^{op},\Set]_{T}$. Using these two facts we prove the following:
\begin{proposition}
If the class $\Gamma$ contains all the finite product cones, then $H$ preserves finite coproducts.
\end{proposition}
We refer to~\cite{Malherbe2010} for the details.
From this, we impose that $\Gamma$ contains all the finite product cones. This is another requirement to obtain a model.

\subsection{$F_{T}\dashv G_{T}$ is a monoidal adjunction}
We recall how a monoidal adjoint pair $(F,m)\dashv
 (G,n)$ induces a monoidal structure for the adjunction $F_{T}\dashv
 G_{T}$ associated with the Kleisli construction.

 \begin{lemma}
   Let $F\dashv G$ be a monoidal adjunction, let $T=GF$, and consider
   the Kleisli adjunction $\xymatrix@1{\mathcal{C}\ar@<1ex>[r]^{F_{T}}&
     \mathcal{C}_{T}\ar@<1ex>[l]^{G_{T}}_{\bot}}$ generated by this adjunction.
   Then $\cC_{T}$ is a monoidal
   category and $F_{T}\dashv G_{T}$ is a monoidal adjunction.
 \end{lemma}

 \begin{proof}
   Since $F\dashv G$ is a monoidal adjunction, it follows that $T=GF$
   is a monoidal monad. The result then follows by general properties of monoidal monads and monoidal adjunctions.
 \end{proof}

\subsection{Abstract model of the quantum lambda calculus}\label{MAIN THEOREM OF CHAPTER 4}
Summing up the parts from previous sections, we have the following theorem.
\begin{theorem} \label{the theorem MAIN THEOREM OF CHAPTER 4}
Given categories $\cB$, $\cC$ and $\cD$, and functors $\cB\stt{\Phi}\cC\stt{\Psi}\cD$, satisfying
\begin{itemize}
\item $\cB$ has finite products, $\cC$ and $\cD$ are symmetric monoidal,
\item $\cB$, $\cC$, and $\cD$ have coproducts, which are distributive w.r.t. tensor,
\item $\cC$ is affine,
\item $\Phi$ and $\Psi$ are strong monoidal,
\item $\Phi$ and $\Psi$ preserve coproducts,
\item $\Phi$ is full and faithful,
\item $\Psi$ is essentially surjective on objects,
\item for every $b\in\cB$, $c,c'\in\cC$ we have
$\cC(\Phi(b),c)\times\cC(\Phi(b),c')\cong\cC(\Phi(b),c\otimes c').$
\end{itemize}
Let $\Gamma$ be any class of cones preserved by the opposite tensor
functor, including all
the finite product cones. Let $Lan_{\Phi}$, $\Phi^{*}$, $F$ and $G$ be defined as
 in Section~\ref{OUTLINE OF THE PROCEDURE} and subsequent sections. Then
 \[\xymatrix{
[\cB^{op},\Set]\ar@<1ex>[rr]^{Lan_{\Phi}}&&[\cC^{op},\Set]\ar@<1ex>[rr]^{F}\ar@<1ex>[ll]^{\Phi^{*}}_{\bot}&&[\cD^{op},\Set]_{\Gamma}\ar@<1ex>[ll]^{G}_{\bot}}\]

 \noindent
 forms an abstract model of the quantum lambda calculus.
\end{theorem}
 \begin{proof} Relevant propositions from
 previous sections.

 \end{proof}

\subsection{Towards a concrete model: constructing $\FinSet\stt{\Phi}\QQ''\stt{\Psi}\QQ$}
\label{Q''CATEGORY AND FUNCTORS PSI AND PHI}\label{Q double prime AND CATEGORY AND FUNCTOR PSI}

The category $\QQ$ of superoperators was defined in Section~\ref{SUPEROPERATORS}.
Here, we discuss a category $\QQ''$
 related to $\QQ$, together with functors $\FinSet\stt{\Phi}\QQ''\stt{\Psi}\QQ$. The goal is to choose $\QQ''$ and the functors $\Phi$ and $\Psi$
 carefully so as to satisfy the requirements of Theorem~\ref{the theorem MAIN THEOREM OF CHAPTER 4}.

Recall the definition of the free affine monoidal category $(\Fwm(\cK),\otimes,I)$:
\begin{itemize}
\item Objects are finite sequences of objects of $\cK$: $\{V_{i}\}_{i\in [n]}=\{V_1,\dots,V_n\}$.
\vspace{.4em}
\item Maps $(\phi,\{f_i\}_{i\in [m]}):\{V_{i}\}_{i\in [n]}\longrightarrow \{W_{i}\}_{i\in [m]}$ are determined by:
\vspace{.6em}
\begin{itemize}
\item[(i)] an injective function $\phi:[m]\rightarrow [n]$,
\item[(ii)] a family of morphisms $f_{i}:V_{\phi(i)}\rightarrow W_{i}$ in the category $\cK$.
\end{itemize}
\vspace{.6em}
\item Tensor $\otimes$ is given by concatenation, with unit  $I$ given by the empty sequence.
\end{itemize}

\begin{proposition}
\label{FREE WEEK SYM MON}
There is a canonical inclusion $\Inc: \cK\rightarrow \Fwm(\cK)$ satisfying:
for any symmetric monoidal category $\cA$ whose tensor unit is terminal and any functor $F:\cK\rightarrow\cA$, there is a unique strong monoidal functor $G:\Fwm(\cK)\rightarrow \cA$, up to isomorphism, such that $G\circ \Inc =F$.
\end{proposition}

We apply this universal construction to the situation where $\cK$ is a discrete category. For later convenience, we let $\cK$ be the discrete category
with finite dimensional Hilbert spaces as objects.
Then $\Fwm(\cK)$ has sequences of Hilbert spaces as objects and dualized, compatible, injective functions as arrows.

Now consider the identity-on-objects inclusion functor $F:\cK\rightarrow \QQ'_s$, where $\QQ'_s$ is the category of simple trace-preserving superoperators defined in Section~\ref{SUPEROPERATORS}. Since $\QQ'_s$ is affine, by Proposition~\ref{FREE WEEK SYM MON}
there exists a unique (up to natural isomorphism) strong monoidal functor $\hat{F}$ such that:
$$\xymatrix@=25pt{
\cK\ar[d]_{\Inc}\ar[r]^{F}&\QQ'_s\\
\Fwm(\cK)\ar[ru]_{\hat{F}}&
}$$

\begin{remark}
\rm
This reveals the purpose of using equality instead of
 $\leq$ in the definition of a trace-preserving superoperator
 (Definition~\ref{TRACE PRESERVING CATEGORY}). When the codomain is the unit,
 there is only one map $f(\rho) = \tr(\rho)$, and therefore $\QQ'_s$ is
 affine.
\end{remark}
Now we remind the reader about the general properties of the free finite coproduct completion $\cC^+$ of a category $\cC$.
The category $\cC^{+}$ has as its objects finite families of objects of $\cC$, say $V=\{V_{a}\}_{a\in A}$, with $A$ a finite set. A morphism from $V=\{V_a\}_{a\in A}$ to $W=\{W_b\}_{b\in B}$ consists of the following:
\begin{itemize}
\item  a function $\phi:A\rightarrow B$,
\item  a family $f=\{f_a\}_{a\in A}$ of morphisms of $\cC$, where
$f_{a}:V_{a}\rightarrow W_{\phi(a)}$.
\end{itemize}
The coproduct in $\cC^{+}$ is just   concatenation of families of objects of $\cC$.
\begin{proposition}
\label{COCOMPLETION OF A CATEGORY}\label{MONO FINITE COMPL}\label{MONOIDAL FUNCTOR PSI}
Given any category $\cA$ with finite coproducts and any functor $F:\cC\rightarrow\cA$, there is a unique finite coproduct preserving functor $G:\cC^{+}\rightarrow \cA$, up to natural isomorphism, such that $G\circ \Inc =F$.
$$\xymatrix@=25pt{
\cC\ar[d]_{\Inc}\ar[r]^{F}&\cA\\
\cC^{+}\ar[ru]_{G}&
}$$
If $\cC$ is a symmetric monoidal category then $\cC^{+}$ is also  symmetric monoidal.
In addition, if we assume that the categories $\cC$ and $\cA$ are symmetric monoidal, then $\Inc$ is a symmetric monoidal functor. If $F$ is a symmetric monoidal functor and tensor distributes over coproducts in $\cA$, then $G$ is a symmetric monoidal functor. Moreover, if $F$ is strong monoidal then so is $G$.
\end{proposition}

In the sequel we want to apply Proposition~\ref{COCOMPLETION OF A CATEGORY} to a concrete category, but first:
\begin{remark}\label{IN AND E EMBEDDING}
\rm

 By definition, $\QQ_s$ is a full subcategory of $\QQ$, and the
 inclusion functor $\In:\QQ_s\rightarrow \QQ$ is strong monoidal.
 Also, since every trace preserving superoperator is trace non-increasing,
 $\QQ'_s$ is a subcategory of $\QQ_s$, and the inclusion functor $E:\QQ'_s\rightarrow \QQ_s$
 is strong monoidal as well.
\end{remark}

We apply the machinery of Proposition~\ref{MONOIDAL FUNCTOR PSI} to the composite functor
$$\Fwm(\cK)\stackrel{\hat{F}}\rightarrow\QQ'_s\stackrel{E}\rightarrow \QQ_s\stackrel{\In}\rightarrow \QQ,$$
where $\In$ and $E$ are as defined in Remark~\ref{IN AND E EMBEDDING}.
\begin{definition}\label{DEFINITION OF THE FUNCTOR PSI}\rm
Let $\QQ''=(\Fwm(\cK))^{+}$ and let $\Psi$ be the unique finite
coproduct preserving functor making the following diagram
commute:
\begin{equation}
\label{THE FUNCTOR PSI}
\xymatrix@=25pt{
\Fwm(\cK)\ar[d]_{\Inc}\ar[r]^{\hat{F}}&\QQ'_s\ar[r]^{E}&\QQ_s\ar[r]^{\In}&\QQ.&\\
(\Fwm(\cK))^{+}\ar[rrru]_{\Psi}&&&
}
\end{equation}
Note that such a functor exists by Proposition~\ref{COCOMPLETION OF A CATEGORY}, and it is strong monoidal.
\end{definition}

\begin{remark}
\rm
Since
$\Psi\{\{V_{i}^{a}\}_{i\in[n_a]}\}_{a\in A}=\coprod_{a\in A}\{(V_1^a\otimes\ldots\otimes V_{n_a}^a)_{*}\}_{*\in 1}$,
the functor $\Psi$ is essentially onto objects. Specifically,
 given any object $\{V_a\}_{a\in A}\in |\QQ|$, we can choose a preimage
 (up to isomorphism) as follows:
\begin{equation}
\Psi\{\{V_{i}^{a}\}_{i\in[1]}\}_{a\in A}=\coprod_{a\in A}\{(V_1^a)_{*}\}_{*\in 1}\cong\{V_a\}_{a\in A}.
\end{equation}
\end{remark}

\begin{lemma}\label{fully faithful strong monoidal functor PHI}
Let $\cC$ be an affine category. Then there is a fully faithful strong monoidal functor  $\Phi:(\FinSet,\times,1)\rightarrow (\cC^{+},\otimes_{\cC^{+}},I)$ that preserves coproducts.
\end{lemma}

\begin{definition}\label{quequeco}\rm
 Recall that $\Fwm(\cK)$ is an affine category and $\QQ''=\Fwm(\cK)^{+}$. Let
$\Phi:\FinSet\rightarrow\QQ''$ be the functor defined by Lemma~\ref{fully faithful strong monoidal functor PHI}.
\end{definition}
\begin{remark}
With the above choice of $\Phi:\FinSet\rightarrow\QQ''$, equation~(\ref{EQUATION MULTIPLICATIVE KERNEL}) in Lemma~\ref{MULTIPLICATIVE KERNEL CONDITIONS} is just the characterization of cartesian products in $\FinSet$ using representable functors.

\end{remark}
\begin{theorem}\label{THE MODEL}
The choice $\cB=\FinSet$, $\cC=\QQ''$,
$\cD=\QQ$, with the functors $\Phi$ as in Definition~\ref{quequeco} and $\Psi$ as in Definition~\ref{DEFINITION OF THE FUNCTOR PSI}, and with $\Gamma$ the class of all finite product cones in $\cD^{op}$, satisfies all the properties required by Theorem~\ref{the theorem MAIN THEOREM OF CHAPTER 4}.
\end{theorem}

\subsection{A concrete model}\label{A CONCRETE MODEL}
\begin{theorem}
 Let $\QQ$, $\QQ''$, $\Phi$, $\Psi$, and $\Gamma$ be defined as in Sections~\ref{SUPEROPERATORS} and~\ref{Q''CATEGORY AND FUNCTORS PSI AND PHI}. Then
\[\xymatrix{
[\FinSet^{op},\Set]\ar@<1ex>[rr]^{Lan_{\Phi}}&&[(\QQ'')^{op},\Set]\ar@<1ex>[rr]^{F}\ar@<1ex>[ll]^{\Phi^{*}}_{\bot}&&[\QQ^{op},\Set]_{\Gamma}\ar@<1ex>[ll]^{G}_{\bot}}\]
  forms a concrete model of the quantum lambda calculus.
 \end{theorem}
\begin{proof}
 This follows from  Theorems~\ref{the theorem MAIN THEOREM OF CHAPTER 4} and~\ref{THE MODEL}.
 \end{proof}
\section{Conclusions and future work}

 We have constructed mathematical (semantic) models
 of higher-order quantum computation, specifically for the
 quantum lambda calculus of Selinger and Valiron. The central idea of
 our model construction was to apply the presheaf construction to a
 sequence of three categories and two functors, and to find a set of
 sufficient conditions for the resulting structure to be a valid
 model. The construction depends crucially on properties of presheaf
 categories, using Day's convolution theory and the  Kelly-Freyd notion of continuity of functors.

 We then identified specific base categories and functors that
 satisfy these abstract conditions, based on the category of
 superoperators. Thus, our choice of base categories ensures that the
 resulting model has the ``correct'' morphisms at base types, whereas
 the presheaf construction ensures that it has the ``correct''
 structure at higher-order types.

 Our work has concentrated solely on the existence of such a model.
 One question that we have not yet addressed is specific properties of
 the interpretation of quantum lambda calculus in this model. It would
 be interesting, in future work, to analyze whether this particular
 interpretation yields new insights into the nature of higher-order
 quantum computation, or to use this model to compute properties of
 programs.

 \paragraph{\bf Acknowledgements.} This research was supported by the
 Natural Sciences and Engineering Research Council of Canada (NSERC)
 and by the Program for the Development of Basic Sciences, Uruguay
 (PEDECIBA).

\end{document}